\documentclass[11pt,a4paper]{amsart}
\usepackage{fullpage,amssymb,amsmath,amsfonts,hyperref,amsthm,paralist,graphicx,color,float, mathtools}
\usepackage{here}
\usepackage{todonotes}
\usepackage[normalem]{ulem}
\usepackage{dsfont}
\usepackage{comment}

\newtheorem{theorem}{Theorem}
\newtheorem{lemma}[theorem]{Lemma}
\newtheorem{prop}[theorem]{Proposition}
\newtheorem{claim}[theorem]{Claim}
\newtheorem{corollary}[theorem]{Corollary}
\newtheorem{observation}{Observation}
\newtheorem*{conjecture}{Conjecture}

\newcommand{\XSays}[3]{{\color{#2}
      {$\rule[-0.12cm]{0.2in}{0.5cm}$\fbox{\tt
            #1:} }%
      \itshape #3
      \marginpar{\color{#2}\tt #1}%
      \def\comment{#3}\def\empty{}\ifx\comment\empty\else
      {$\rule[0.1cm]{0.3in}{0.1cm}$\fbox{\tt
            end}$\rule[0.1cm]{0.3in}{0.1cm}$} \fi
   }%
}

\newcommand{\yop}{{}^{\circ}Y}

\newcommand{\ycl}{{}^{\bullet}Y}

\newcommand{\al}[1]{\begin{align*}#1\end{align*}}

\newcommand{\p}{\mathbb P}
\newcommand{\E}{\mathbb E}
\newcommand{\cs}{\mathcal S}
\newcommand{\I}{\mathcal I}
\newcommand{\taur}{\overset{\rightharpoonup}{\tau}}
\newcommand{\taul}{\overset{\leftharpoonup}{\tau}}

\newcommand{\thetac}{{}^{\bullet}\theta}
\newcommand{\thetao}{{}^{\circ}\theta}

\makeatletter

\newcommand\floatc@mybox[2]{\vbox{\hbadness10000
\moveleft3.4pt\vbox{\advance\hsize by6.8pt
\hrule \hbox to\hsize{\vrule\kern3pt
\vbox{\kern3pt\vbox{\advance\hsize by-6.8pt{\@fs@cfont #1} #2}\kern3pt}\kern3pt\vrule}}}}%
\newcommand\fs@mybox{\def\@fs@cfont{\bfseries}\let\@fs@capt\floatc@mybox
\def\@fs@pre{\setbox\@currbox\vbox{\hbadness10000
\moveleft3.4pt\vbox{\advance\hsize by6.8pt
\hrule \hbox to\hsize{\vrule\kern3pt
\vbox{\kern4.5pt\box\@currbox\kern4.5pt}\kern3pt\vrule}\hrule}}}%
\def\@fs@mid{}%
\def\@fs@post{}%
\let\@fs@iftopcapt\iftrue}
\makeatother
\floatstyle{mybox}
\newfloat{inset}{tb}{ins}
\floatname{inset}{Inset}

\date{\today}
\title{Phase transition for a non-attractive infection process in heterogeneous environment}
\author{Marinus Gottschau}
\address[Marinus Gottschau]{TUM School of Management and Department of Mathematics, Technische Universit\"at M\"unchen, Arcisstra\ss e 21, 80333 M\"unchen, Germany}
\email{marinus.gottschau@tum.de}
\author{Markus Heydenreich}
\author{Kilian Matzke}
\address[Kilian Matzke, Markus Heydenreich]{Mathematisches Institut, Ludwig-Maximilians-Universit\"at M\"unchen, Theresienstra\ss e 39, 80333 M\"unchen, Germany}
\email{matzke@math.lmu.de, m.heydenreich@lmu.de}
\author{Cristina Toninelli}
\address[Cristina Toninelli]{CNRS, Laboratoire de Probabilit\'es et Mod\`eles Al\'eatoires, Univ.\ Paris VI et VII, 
	B\^atiment Sophie Germain, Case courrier 7012, 75205 Paris Cedex 13, France}
\email{cristina.toninelli@upmc.fr}
\thanks{C.Toninelli acknowledges the support by the ERC Starting Grant 680275
  MALIG and  ANR-15-CE40-0020-03} 

\keywords{Contact process, phase transition, survival versus extinction}
\subjclass[2010]{82C22, 82C26}

\begin{document}
\maketitle
\begin{abstract}
We consider a non-attractive three state contact process on $\mathbb Z$ and prove that there exists a regime of survival as well as a regime of extinction. In more detail, the process can be regarded as an infection process in a dynamic environment, where non-infected sites are either healthy or passive. Infected sites can recover only if they have a healthy site nearby, whereas non-infected sites may become infected only if there is no healthy and at least one infected site nearby. The transition probabilities are governed by a global parameter $q$: for large $q$, the infection dies out, and for small enough $q$, we observe its survival. The result is obtained by a coupling to a discrete time Markov chain, using its drift properties in the respective regimes.
\end{abstract}


\section{Introduction}
\subsection{History}

The classical contact process, as introduced by Harris in 1974 \cite{harris74}, has been a central topic of research in interacting particle systems. It is formally defined as $\{0,1\}^{\mathbb Z^d}$-valued spin system, where 1's flip to 0's at rate 1, and flips from 0 to 1 occur at rate 
$\lambda$ times the number of neighbors in state 1, 
where $\lambda>0$ is a parameter of the model. Commonly, the lattice sites are called `individuals', which are either \emph{infected} (i.e., in state~1) or \emph{healthy} (i.e., in state~0).  Many fundamental questions have been settled for this model, the results are summarized in the monographs by Liggett and Durrett in \cite{Durrett88, Liggett85, Liggett99}.

Among the most important results are the existence of a phase transition for survival of a single infected particle, the complete convergence theorem, and extinction of the critical contact process. Much more refined results have appeared in recent years. 
In view of these successes, it may seem surprising that results are considerably sparse as soon as multitype contact processes are considered. Results have only be achieved in very specific situations, examples are the articles by Cox and Schinazi \cite{CoxSchin09}, Durrett and Neuhauser \cite{DurreNeuha97}, Durrett and Swindle \cite{DurreSwind91}, Konno et al.\ \cite{KonnoSchinTanem04}, Neuhauser \cite{Neuhauser92}, and Remenik \cite{Remen08} for various models.

Our focus here is on the contact process with three types, and this carries already severe complications. A fair number of models considered in the literature stems from a biological context (either evolvement of biological species or vegetation models); typical questions that have been considered are coexistence versus extinction and phase transitions. Examples are the work of Broman \cite{Broman07} and Remenik \cite{Remen08}.

There are two features that are shared by all of these models: they are monotonic and they are (self-)dual (we refer to \cite{Liggett99} for a definition of these terms). 
These two properties are crucial ingredients in the analysis; if they fail, then most of the known tools fail. This might be illustrated by looking at Model A in \cite{BergBjornHeyde15}, which is a certain 3-type contact process. Even though there are positive rates for transitions between the various states of this model and apparent monotonicity, the lack of any usable duality relation prevented all efforts in proving convergence to equilibrium for that model. 

For the model considered in the present paper, it appears that there is no duality relation that we can exploit and monotonicity is restricted to a very particular situation only. Yet we are able to prove the occurrence of a phase transition by means of coupling to certain discrete-time Markov chains and analyzing drift properties of these chains. 
We believe that the technique presented here is useful in greater generality. 
A motivation for studying this process stems from the connection with the out of equilibrium dynamics of kinetically constrained models, as we will explain in detail in Section \ref{discussion}.  We believe that the proof techniques apply in similar situations. 


\subsection{The model}

Our state space is $\Omega=\{0,1,2\}^{\mathbb Z}$, equipped with the product topology (which makes $\Omega$ compact). Further, $q\in[0,1]$ is a parameter and $(\eta_t)_{t\ge 0}$ is a Markov process on $\Omega$. We say that at time $t$,
\al{ \text{site }x \text{ is} \begin{cases}
\emph{healthy} &\mbox{if } \eta_t(x)=0, \\ \emph{passive} &\mbox{if } \eta_t(x)=1 \text{ and} \\ \emph{infected} &\mbox{if } \eta_t(x)=2.
\end{cases}}
Informally, we can describe the process as follows. Each site $x$ independently waits an exponential time with intensity 1 and then updates its state according to the following rules:
\begin{itemize}
	\item If at least one neighboring site is healthy, then $x$ becomes healthy with probability $q$ and passive w.p.~$1-q$.
	\item If at least one neighbor is infected and none is healthy, then a previously healthy $x$ becomes infected w.p.~$1-q$, a previously passive $x$ becomes infected w.p.~$q$ and remains in its state otherwise.
\end{itemize}
For a more formal description, the process can be characterized by its probability generator, which is the closure of the operator 
\al{\mathcal Lf(\eta) =   \sum_{x \in \mathbb Z} \Big[ & c_x(\eta) q (f(\eta^{x,0})- f(\eta))  + c_x(\eta) (1-q) (f(\eta^{x,1})- f(\eta)) \\
					&+ \bar c_x(\eta) \left[ q \, \mathds 1_{\{\eta(x) = 1 \}} + (1-q) \mathds 1_{\{\eta(x) =0 \}}\right] (f(\eta^{x,2})- f(\eta)) \Big],} 
\[f\in \big\{f:\Omega\to\mathbb R\text{ cont.}\colon \lim_{x\to\infty}\sup\{|f(\eta)-f(\eta')|\colon \eta,\eta'\in\Omega, \eta(y)=\eta'(y) \text{ for all }y\neq x\}=0\big\}.\]
Here, $c_x(\eta) = \mathds 1_{\{ \eta(x-1) \cdot \eta(x+1) =0 \}}$ and $\bar c_x(\eta) = \mathds 1_{\{ \eta(x-1) \cdot \eta(x+1) \geq 2\}}$. Furthermore, $\eta^{x,i}$ is the configuration where $\eta^{x,i}(y) = \eta(y)$ for all $y \neq x$ and $\eta^{x,i}(x) = i$, $x,y \in \mathbb Z, i \in \{0,1,2\}$. 

For an initial configuration $\eta\in\Omega$, we denote by $\mathbb P^\eta$ the corresponding probability measure. This superscript will be dropped for the sake of convenience if context permits.

As we wrote earlier, monotonicity is an important tool in the analysis of such processes. One monotonicity property the $(\eta_t)$ process exhibits is the following. 
\begin{claim} \label{claim:monoinf}
For arbitrary $\eta \in \Omega$ and $x \in \mathbb Z$, we have that
		\al{\p^{\eta'} \left[\eta_t \notin \{0,1\}^{\mathbb Z} \textrm{ for all } t \geq 0 \right] \geq \p^{\eta''} \left[\eta_t \notin \{0,1\}^{\mathbb Z} \textrm{ for all } t \geq 0\right],}
where $\eta' = \eta^{x,2}$ and $\eta'' \in \{ \eta^{x,1}, \eta^{x,0} \}$. 
\end{claim}
In words, additional infected sites cannot decrease the chance of the infection's survival. However, the same is not necessarily true anymore for $\eta'= \eta^{x,1}$ and $\eta'' = \eta^{x,0}$.
\begin{proof}
	If we couple the two processes with respective initial measures, we claim that, almost surely, $\eta'_t(x)\in \{\eta''_t(x),2\}$ for all $t\geq 0$ and $x\in \mathbb Z$. This is a consequence of the definition of the dynamics and corresponding transition rates.
\end{proof}

\subsection{Results and discussion}\label{discussion}
Our main result is a phase transition for $(\eta_t)$ in the parameter $q$: if $q$ is very close to $0$, then any number of initially infected sites survives with positive probability, whereas if $q$ is close to $1$, then the infection dies out with probability $1$.

\begin{theorem} \label{mainthm}
There exist values $0<q_0 <q_1<1$ such that
\begin{enumerate}
	\item[(i)] for any initial configuration $\eta \notin \{0,1\}^{\mathbb Z}$, we have
			\al {\p^\eta \left[\eta_t \notin \{0,1\}^{\mathbb Z} \textrm{ for all } t \geq 0\right] >0 & \qquad \textrm{for all } q \leq q_0,}
	\item[(ii)] and for any initial configuration $\eta$ with $\sup_{x \in \mathbb Z} \inf_{y \in \mathbb Z}\{|x-y|: \eta(y)=0\} < \infty$, we have
			\al{\p^\eta \left[\eta_t \in \{0,1\}^{\mathbb Z}\right] \xrightarrow{t \to \infty} 1 \textrm{ a.s.} & \qquad \textrm{for all } q \geq q_1.}
\end{enumerate} 

\end{theorem}

We thus prove the existence of different regimes without relying on duality properties. Since there is no monotonicity that can be exploited here, we can not rule out that there are more than one transitions between the regimes ``the infection dies out'' and ``the infection survives''. However, we conjecture the following statement to be true.

\begin{conjecture}
The function $q \mapsto \p^\eta \left[\eta_t \notin \{0,1\}^{\mathbb Z} \textrm{ for all } t \geq 0\right] $ is decreasing in $q\in(0,1]$.
\end{conjecture}

This would imply a critical value $q_c$ such that if $q<q_c$ the infection survives with positive probability, while if $q > q_c$ the infection dies out with probability $1$.

Note that the case $q=0$ is degenerate and of little interest, as it admits traps: If there is a site $x \in \mathbb Z$ and a time $t \geq 0$ such that we exhibit $(\eta_t(x), \eta_t(x+1), \eta_t(x+2)) = (1,0,1)$, then this triple will remain fixed for all $t' \geq t$.

A very related process to the one just introduced is the simpler version for which, informally, the second condition is altered to: ``If at least one neighbor of $x$ is infected and none is healthy, then $x$ becomes infected.'' It is clear that the set of infected sites in this version dominates our process. However, the same proof techniques used below yield similar results to Theorem~\ref{mainthm} (namely, also a phase transition).

\subsection*{Connections to kinetically constrained models.}
This model has an indirect connection with  Frederickson-Andersen 1 spin facilitated model (FA1f) \cite{fa1f, fa1fa, fa1fb}.
In this case, the configuration space is $\{0,1\}^{\mathbb Z}$ and the dynamics are defined as follows: a site $x$ with occupation variable $0$ flips to $1$ at rate $1-q$ iff at least one among its nearest neighbors is in state zero; a site $x$ with occupation variable $1$ flips to $0$ at rate $q$ iff at least one among its nearest neighbors is in state zero. Note that the constraint for the $0\to 1$ and the $1\to 0$ updates are the same and the dynamics satisfies detailed balance w.r.t.~the product measure $\mu$ with $\mu(\eta(x)=0)=q$.
Note also that the dynamics of our contact process coincide with the FA1f dynamics if we start from a configuration which does not contain infected sites.

A non trivial problem for FA1f dynamics is to determine convergence to the equilibrium measure $\mu$ for some reasonable initial measure, e.g.~an initial product measure with density of healthy sites different from $q$ \cite{fa1f}. We will now explain how our results provide an alternative approach to prove convergence to equilibrium in a restricted density regime. A possible strategy to prove convergence to equilibrium for FA1f dynamics started from an initial configuration  $\eta_0$ is to couple it with some $\tilde \eta_0$ distributed according to $\mu$. This gives rise to  a process with 4 states $\{0,1,2^{\downarrow},2^{\uparrow}\}$. Here, $0$ represent sites where both configurations are $0$;
$1$ sites where both configurations are $1$; $2^{\downarrow}$ sites where $\eta$ is $0$ and $\tilde\eta$ is 1; and $2^{\uparrow}$
 sites where $\eta$ is $1$ and $\tilde\eta$ is $0$. 
If we now denote the union of sites in state $2^{\downarrow}$ and  $2^{\uparrow}$ as "infected sites", then if infection dies out, the original process started in $\eta_0$ is distributed with the equilibrium measure (since there are no more discrepancies with the process evolved from $\tilde\eta$ which is at equilibrium at any time). 
It is not difficult to verify that the dynamics of the $4$ state contact process induced by the standard coupling among two configurations evolving with FA1f dynamics are such that the union of sites in state $2^{\downarrow}$ and $2^{\uparrow}$ is dominated by the infected sites of our $3$-state contact process.
Thus when infection dies out for our process it also dies out for the $4$-state contact process and from our  
Theorem~\ref{mainthm} (ii) we get convergence to equilibrium for $q\geq q_1$ for the FA1f dynamics.
This result was already proven by a completely different technique in Blondel et al.\ \cite{fa1f} for parameter $q > 1/2$.
Notice that convergence to equilibrium is expected to hold for FA1f dynamics at all $q>0$ starting from $\eta$ satisfying the hypothesis of our Theorem~\ref{mainthm}~(ii), namely infection should always disappear in the $4$ state contact process. This is certainly not the case for our $3$ state contact process which has a survival extinction transition, as proved by  Theorem~\ref{mainthm}~(i).


\section{The small $q$ regime} \label{sec:smallq}

In this section, we prove assertion (i) of Theorem~\ref{mainthm}. First, We define	
		\al{\Omega^* = \{ \eta \in \Omega: \exists \, a \leq b \in \mathbb Z: \{x: \eta(x)=2\} = [a,b] \cap \mathbb Z \},}
the set of configurations where infected sites form a finite, nonempty interval.
\begin{prop} \label{thm:smallq}
Consider some $\eta \in\Omega^*$. Then there exists $0<q_0$ such that
			\al {\p^\eta \left[\eta_t \notin \{0,1\}^{\mathbb Z} \textrm{ for all } t \geq 0\right] >0 & \qquad \textrm{for all } q \leq q_0,}
\end{prop}

For the proof, we observe first that the set of sites in state 2, which we call the \emph{infected cluster}, is always connected. We would like to focus on the behavior of the infected boundary sites and so, due to symmetry, on \al{\I(t) := \sup \{x \in \mathbb Z: \eta_t(x)=2 \},} the position of the rightmost infected site. If there is only one infected site (thus, leftmost and rightmost infected site coincide), both with positive probability the next change in number of infected sites might result in zero (extinction of the infection) or two infected sites. If the number of infected sites is at least two, only the status on the sites to the right of the rightmost infected site have direct influence on the `movement' of $\I(t)$.

In (an informal) summary, if the infection shrinks to size one, it recovers with positive probability to size at least two. If we show that from there, infection spreads with positive probability, we obtain our result. Therefore, we focus on this latter regime in the following.

With this in mind, we now introduce a Markov chain, which can be interpreted as a simplified model of the rightmost infected site and its local right neighborhood, and prove a drift property for it. This shall turn out to be useful when coupling this auxiliary Markov chain to our original process in section \ref{coupling}. 

\subsection{An auxiliary Markov chain}

We define a Markov chain $(Y_i)_{i \geq 0}$ living in the (countable) state space $\cs = \mathbb Z \times \{0,1\}^3$. We denote its first coordinate as the chain's \emph{level} or \emph{state} and thus can partition $\cs$ into its $n$-states
	\al{\cs_n:=\{ (\omega_1, \omega_2, \omega_3, \omega_4) \in \cs: \omega_1 = n\}}
for any integer $n$. The Markov chain is defined by its transition graph shown in Figure~\ref{fig:ychain}. The subgraphs induced by $\cs_n$ are isomorphic, and furthermore, two states from $\cs_n$ and $\cs_m$ for $|m-n| \geq 2$ have transition probability zero. 
Hence, for simplicity, we can restrict ourselves to depicting the transition graph induced by $\cs_n$, with additional states in $\cs_{n \pm 1}$ along with their respective transition probabilities.
We denote the probability measure of this Markov chain by $\p$ 
(we trust that this causes no confusion with the measure of the interacting particle system). 


\begin{figure}
\centering
  \includegraphics[width=\textwidth]{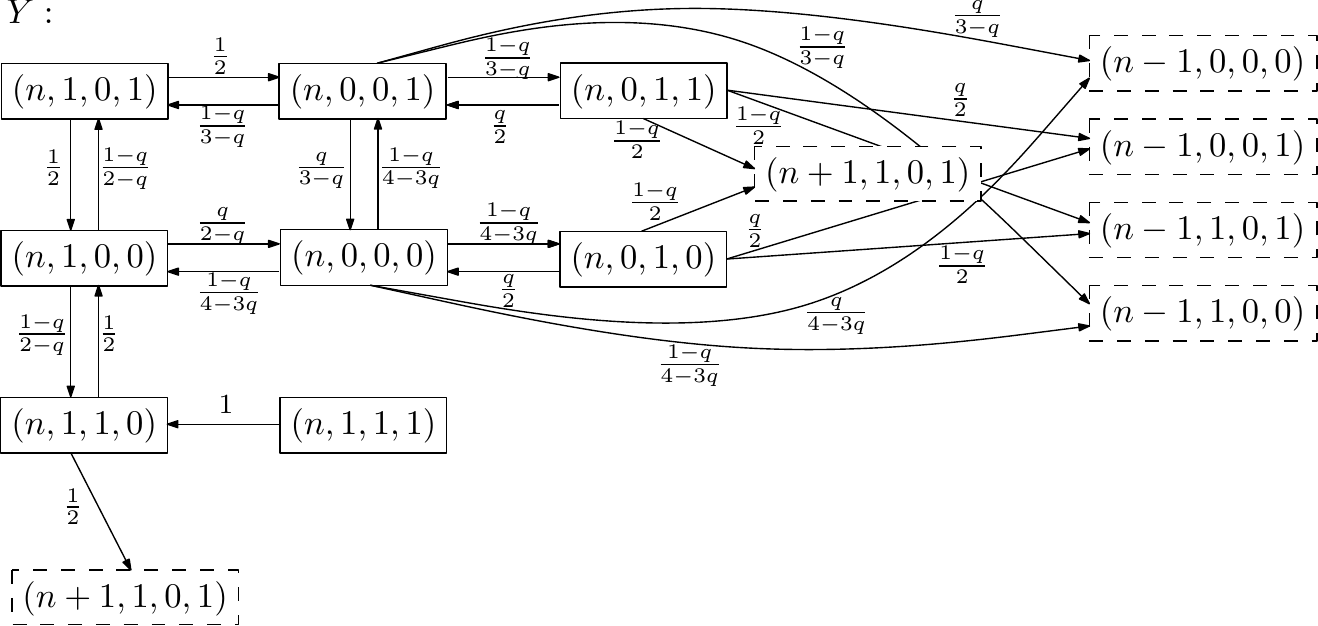}
  \caption{Transition subgraph of $Y$ induced by $\cs_n$ and its neighboring states.}
  \label{fig:ychain}
\end{figure}

We define the stopping time $\tau$ to be the first time the Markov chain changes its level:
		\al{\tau := \min\{i \in \mathbb N: \exists n \in \mathbb Z: Y_0 \in \cs_n, Y_i \in \cs_{n\pm 1} \}.}
Using $Y_i^m$ (for $0 \leq m \leq 3$) to access the $m$th component of the state which $Y$ is in at time $i$, say that $Y_{i}$ is a progressive step (progress) if $Y^1_i = Y^1_{i-1} + 1$ and similarly call $Y_i$ a regressive step (regress) if $Y^1_i = Y^1_{i-1} - 1$. We say that a natural number $i$ is a step time (step) if $Y_i$ is either a progressive or a regressive step.

\begin{lemma}\label{lem:smallqonestep}
Let $n \in \mathbb Z$ and $\varepsilon>0$. Then there exists $0<q_0<1$ such that
		\al{\frac{1}{2} - \varepsilon &< \p[Y_\tau \in \cs_{n+1} \mid Y_0 = (n,1,0,1)] < \frac{4}{7} + \varepsilon , \\
				\frac{2}{3} - \varepsilon &< \p[Y_\tau \in \cs_{n+1} \mid Y_0 = (n,1,0,0)] }
for all $q<q_0$.
\end{lemma}

\begin{proof}
The proof proceeds by counting paths in the transition graph. We define
		\al{ \theta_1 &:= \p[Y_\tau \in \cs_{n+1} \mid Y_0 = (n,1,1,0)], \\
				\theta_2 &:= \p[Y_\tau \in \cs_{n+1} \mid Y_0 = (n,1,0,0)], \\
				\theta_3 &:= \p[Y_\tau \in \cs_{n+1} \mid Y_0 = (n,1,0,1)].}
We also set $a:= \tfrac{1-q}{2(3-q)}$ to be the \emph{weight} of the 2-cycle between states $(n,1,0,1)$ and $(n,0,0,1)$. The weight of a cycle is the probability that the Markov chain transitions along this cycle in the transition graph. As a path may use this cycle arbitrarily often, we have
		\al{ \theta_1 &= \frac{1}{2} \left(1+\theta_2\right), \\
					\theta_2 &\geq \frac{1-q}{2-q} \theta_1 + \frac{1-q}{2-q} \theta_3, \\
					\theta_3 &\geq \left(\frac{1}{2} \theta_2 + \frac{1-q}{2} a \right) \sum_{k\geq 0} a^k = \frac{1}{1-a}\left(\frac{1}{2} \theta_2 + \frac{1-q}{2} a \right),}
using the strong Markov property. This leads to the explicit lower bounds
		\al{ \theta_1 & \geq \frac{15-9q+3q^2-q^3}{18-2q^2}, \\
					\theta_2 &\geq \frac{6-9q+4q^2-q^3}{9-q^2}, \\
					\theta_3 &\geq \frac{3-4q+q^2}{6+2q}.}
For small $q$, all of these values are strictly larger than $\tfrac{1}{2}$, except for $\theta_3$, where we have $\theta_3 \nearrow \tfrac{1}{2}$ as $q \to 0$. Finally set $b:= \tfrac{1-q}{2(2-q)}$ to be the weight of a 2-cycle between states $(n,1,0,1)$ and $(n,1,0,0)$ and observe that, by counting paths ending in $\cs_{n-1}$, we have
		\al{1-\theta_3 \geq \left(\frac{1}{2(3-q)}+\frac{1-q}{4(3-q)}\right) \left(\sum_{n\geq 0} \sum_{k=0}^n \binom{n}{k} a^k b^{n-k} \right) = \frac{6-5q+q^2}{14-6q}.}
In the first parenthesis, the first term comes from paths ending in $(n-1, 0,0,0)$ and $(n-1,1,0,0)$, whereas the second terms comes from paths ending in $(n-1,0,0,1)$ as well as $(n-1,1,0,1)$. The lemma follows for $q$ sufficiently small.

\end{proof}


\begin{lemma}\label{lem:smallqdrift}
There exists $0<q_0<1$ such that for all $0<q<q_0$, we have \al{\E\left[Y^1_\tau - Y^1_0\right]>0.}
\end{lemma}
\begin{proof}
We start by defining \al{\tau_2 := \min\{i > \tau: \exists n \in \mathbb Z: Y_\tau \in \cs_n, Y_i \in \cs_{n\pm 1} \}}
to be the first level change after $\tau$ and actually prove $\E\left[Y^1_{\tau_2} - Y^1_0\right]>0$. Noting that after two level changes, $Y^1$ will either have increased or decreased by 2 or not changed at all, the lemma follows from proving
		\begin{align} \label{eq:progdrift}
		\p\left[Y_{\tau_2} \in \cs_{n+2} \mid Y_0 \in \cs_n\right] > \p\left[Y_{\tau_2} \in\cs_{n-2} \mid Y_0 \in \cs_n\right].
		\end{align}
for any integer $n$. To this end, we make the following observation, which is an immediate consequence of the definition of the Markov chain dynamics.

\begin{observation}\label{obs:stepstate}
Let $Y_0 \in \cs_n$. Then from $Y_\tau \in \cs_{n+1}$, it follows that $Y_\tau = (n+1,1,0,1)$. On the other hand, if $Y_\tau \in \cs_{n-1}$, then $Y_\tau$, with probability $q$, is one of the two states $\{(n-1,0,0,1), (n-1,0,0,0)\}$ and, with probability $1-q$, is one of the two states $\{(n-1,1,0,1), (n-1,1,0,0)\}$.
\end{observation}

We can thus restrict ourselves to proving \eqref{eq:progdrift} for $Y_0$ being one of the two `\emph{good}' $n$-states $\mathcal G_n :=\{(n,1,0,1), (n,1,0,0)\}$, as we are allowed to choose $q_0$ sufficiently small. Combining Observation~\ref{obs:stepstate} with Lemma~\ref{lem:smallqonestep}, we have
		\al{\tilde \alpha := \p\left[Y_{\tau_2} \in \mathcal S_{n+2} \mid Y_0 \in \mathcal G_n \right] \geq \min_{\omega \in \mathcal G_n} \left(\p\left[Y_\tau \in \mathcal G_{n+1} \mid Y_0 = \omega \right] \right)^2 > \left(\frac{1}{2}-\varepsilon\right)^2}
for some $\varepsilon>0$ and $q$ appropriately small. Recalling that $a:= \tfrac{1-q}{2(3-q)}$ was the weight of a 2-cycle between states $(n,1,0,1)$ and $(n,0,0,1)$ and $b:= \tfrac{1-q}{2(2-q)}$ the value of a 2-cycle between states $(n,1,0,1)$ and $(n,1,0,0)$, and setting $\omega = (n,1,0,1)$,  $\omega' = (n,0,0,1)$ as well as $\omega'' = (n-1,1,0,0)$, we have
		\al{\kappa &:= \p\left[Y_\tau \in \mathcal G_{n-1}, Y_{\tau-1} = \omega' \mid Y_0 = \omega \right] \\
				& \geq \frac{1-q}{2(3-q)} \left(\sum_{m \geq 0} \sum_{k=0}^m \binom{m}{k} a^kb^{m-k} \right) = \frac{1-q}{2(3-q)} \cdot \frac{1}{1-a-b} \\
				& = \frac{2-3q+q^2}{7-3q},}
with the bound obtained simply by counting paths from $\omega$ to $\omega''$ which pass through $\omega'$ in their second to last step.  With $\varepsilon$ small enough ($\varepsilon<1/100$ say), we are now able to bound $\alpha:=\p\left[Y_{\tau_2} \in \cs_{n-2} \mid Y_0 = \omega \right]$, the probability of double regress from $\omega$, as follows:
		\al{ \alpha &  = \p\left[\{Y_{\tau_2} \in \cs_{n-2} \} \cap \{Y_\tau \in \mathcal \cs_{n-1}\} \mid Y_0 = \omega \right] \\
			& \leq q + (1-q) \cdot \p\left[\{Y_{\tau_2} \in \cs_{n-2} \} \cap \{Y_\tau \in \mathcal G_{n-1}\} \cap \{Y_{\tau-1} = \omega' \} \mid Y_0 = \omega \right] \\
			& \quad + q + (1-q) \cdot \p\left[\{Y_{\tau_2} \in \cs_{n-2} \} \cap \{Y_\tau \in \mathcal G_{n-1}\} \cap \{Y_{\tau-1} \neq \omega' \} \mid Y_0 = \omega \right] .}
Rearranging by defining $B = \{Y_\tau \in \mathcal G_{n-1}\} \cap \{Y_{\tau-1} = \omega' \}$ and $B' = \{Y_{\tau} \in \mathcal G_{n-1}\} \cap \{Y_{\tau-1} \neq \omega' \}$ and observing that the event $B$ implies that $Y_\tau=\omega''$, we continue to find that
		\al{ \alpha & \leq  2q + (1-q) \sum_{A \in \{B, B'\}} \p \left[Y_{\tau_2} \in \cs_{n-2} \mid A \right] \cdot \p \left[A \mid Y_0 = \omega \right] \\
			& \leq 2q + (1-q) \cdot \p\left[Y_{\tau_2} \in \cs_{n-2} \mid Y_\tau = \omega'' \right] \cdot \kappa \\
			& \quad + (1-q) \cdot \p\left[Y_{\tau_2} \in \cs_{n-2} \mid Y_\tau = \omega \right] \cdot \left( \p\left[Y_\tau \in \mathcal G_{n-1} \mid Y_0 = \omega \right] - \kappa \right) \\
			& \leq 2q + (1-q) \left( \kappa (1- \theta_2) + (1-q)(1-\theta_3)(1- \theta_3 - \kappa) \right) \\
			& < 2q + (1-\theta_3)^2 + \kappa (\theta_3 - \theta_2) \\
			& < 2q + \left(\frac{1}{2} + \varepsilon\right)^2 + \left(\frac{2}{7} - \varepsilon\right) \left( \frac{4}{7}- \frac{2}{3}\right) \\
			& < 2q + \tilde \alpha - \frac{4}{21} \left( \frac{1}{7} -  11 \varepsilon \right) < \tilde \alpha}
for our chosen $\varepsilon$ and $q$ sufficiently small, where $\theta_i$ have been defined in the proof of Lemma \ref{lem:smallqonestep}. Note that again we make heavy use of the strong Markov property as well as the bounds from Lemma~\ref{lem:smallqonestep}.
\end{proof}

\subsection{The coupling}
\label{coupling}

We are now ready to return to our process. Recall that $Y$ should be thought of as a model of the right neighborhood of the rightmost infected site in the original process. Intuitively speaking, we want to find a coupling such that $Y^0 \leq \I(t)$ at any given time---this, however, is ill-defined. To make it more precise, let us first formally build towards the discrete version of the segment of the process that is of interest (i.e., the right neighborhood of the rightmost infected site). For $(\eta_t)_{t \geq 0}$ a realization of the process in $\Omega^*$, we define the map $\Phi:\Omega^* \to \mathbb Z \times \{0,1\}^4$ as
		\al{\Phi(\eta_t) = \left(\I(t), \left(\eta_t(\I(t)+i)\right)_{i=1}^4\right).}
Hence, $(\Phi_t)_{t \geq 0} = (\Phi(\eta_t))_{t \geq 0}$ is the segment of the process we are interested in. Let $S(x)=(s_i(x))_{i \in \mathbb N}$ be the sequence of clock rings for site $x$. That is, $s_1 \sim \textrm{Exp}(1)$ and $\left(s_{i+1}(x)-s_i(x)\right) \sim \textrm{Exp}(1)$ for all $i \in \mathbb N$. This allows us to define $(R_i)_{i \in \mathbb N_0}$, the sequence of times of clock rings of the process restricted to $\left(\eta_t(\I(t)+i)\right)_{i=0}^4$, as $R_0 =0$ and
		\al{R_{i+1} = \inf\left\{s_j(x): s_j(x) > R_i, x \in \{\I(R_i) + l: 0 \leq l \leq 4 \}, j\in \mathbb N \right\}} 
for all $i \geq 0$. We are interested in the process $(X_i)_{i \in \mathbb N_0}$, a subset of $(\overline X)_{i \in\mathbb N_0}$, where $ \overline X_0 = X_0=\Phi(\eta_{R_0})$ and
		\al{\overline X_i &= \Phi(\eta_{R_i}), \\ X_i &= \Phi(\eta_{R_l}), \quad l= \inf \{ k \geq i:  \Phi(\eta_{R_k}) \neq \Phi(\eta_{R_{i-1}}) \}}
for all $i \geq 1$. In words, $(\overline X)_i$ is the embedded discrete time chain of $(\Phi)_t$, and $X$ is the chain obtained from $\overline X$ by removing all of the self-loops. Process $X$ is the one which, in certain time windows, behaves very much like $Y$. To make this precise, we define $(R'_i)_{i \in\mathbb N_0}$ with $R'_0=0$ and
		\al{R'_{i+1}= \inf\left\{\left\{t> R'_i: \I(R'_i)\neq \I(t)\right\} \cup \{ s_j(\I(R'_i) + 4):  s_j(\I(R'_i) + 4)> R'_i, j\in \mathbb N \} \right\}}
to be the times when either the position of the rightmost infected particle changes or the clock at the site determining the boundary condition, $\I(\,\cdot\,) + 4$, rings. 
We call $W_i := [R'_i, R'_{i+1})$ the \emph{stable windows} for all $i \geq 0$. A stable window closes whenever the boundary site rings or the infected site moves. We can now proceed to describe the behavior of $X$ in a stable window. As we did for $Y$, we can partition the state space of $X$ into its levels and, as no confusion arises this way, call them $\cs_n$ as well. The dynamics within $W_i$ depend only on $\Phi(\eta_{R'_i})$, namely the initial state also encoding the boundary conditions, and are therefore Markovian. Given this initial state for $W_0$, we can depict the transition graph in a very similar way as the one for $Y$, as the subgraphs induced by the levels are again isomorphic, the states in neighboring levels are terminal as they `close' the window $W_0$. Conditional on the boundary conditions, the two transition graphs are shown in Figures~\ref{fig:xchain1} and \ref{fig:xchain0}.

\begin{figure}
\centering
  \includegraphics[width=.8\linewidth]{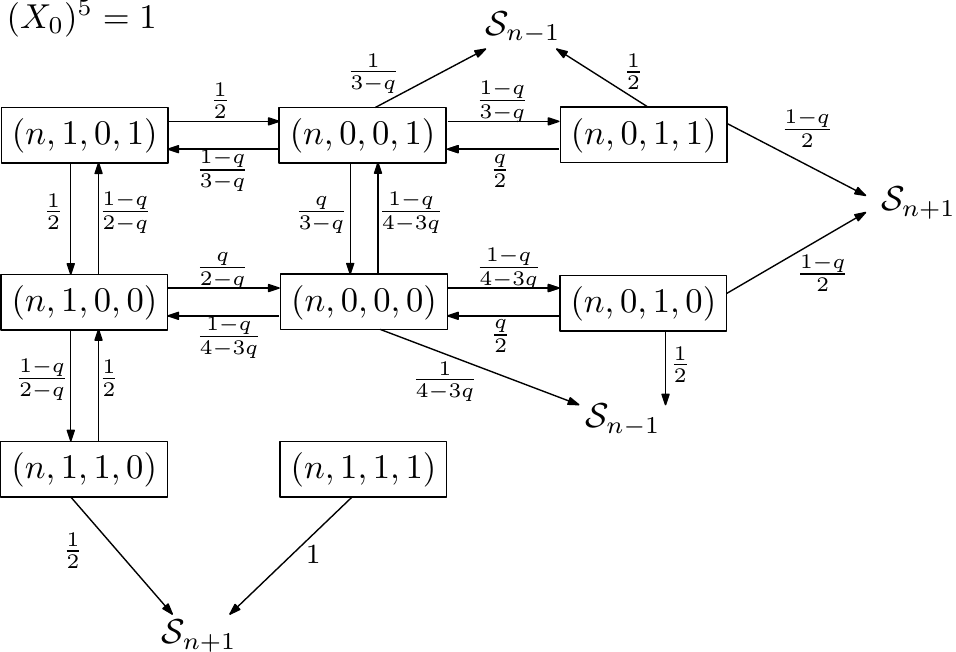} \ \\
  \caption{Transition subgraph of $X$ in window $W_0$ induced by $\cs_n$ for passive initial boundary conditions. Note that entering $\cs_{n\pm1}$ closes $W_0$.}
  \label{fig:xchain1}
\end{figure}

\begin{figure}
\centering
  \includegraphics[width=.8\linewidth]{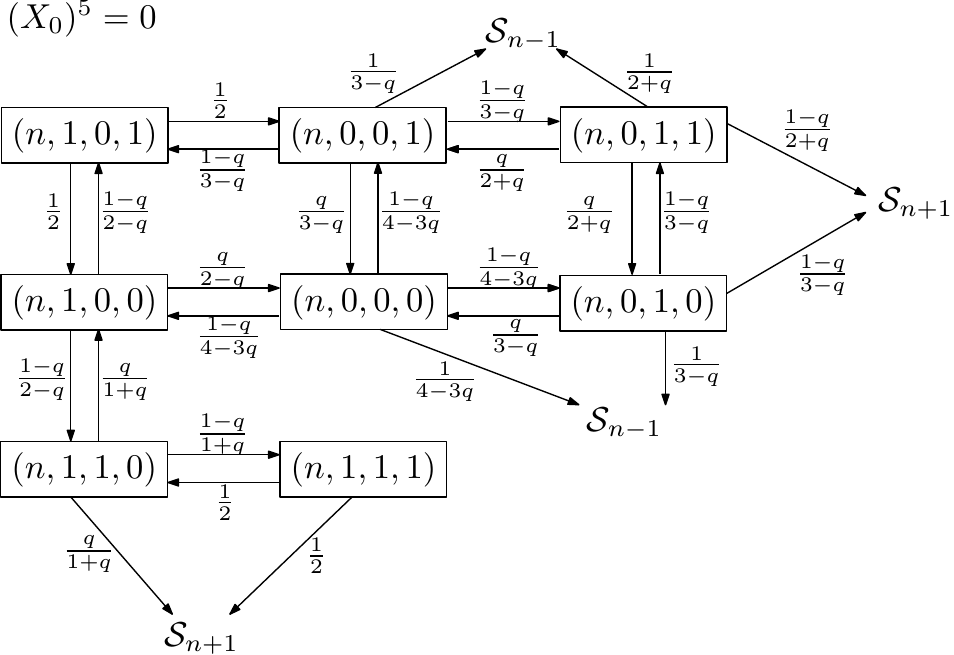} \ \\
  \caption{Transition subgraph of $X$ in window $W_0$ induced by $\cs_n$ for healthy initial boundary conditions.}
  \label{fig:xchain0}
\end{figure}

\begin{lemma} \label{lem:coupling}
We have that for any $i\geq1$
			\al{\p \left[X_{R'_1} \in \cs_{n+1} \mid X_{R'_0} \in \cs_n, X_{R'_1} \notin \cs_n  \right] \geq \p \left[Y_\tau \in \cs_{n+1} \mid Y_0 \in \cs_n  \right]}
for any integer $n$. In words, conditioned on $W_0$ closing due to a level change, the probability of progress in $X$ is bounded from below by the probability of progress in $Y$.
\end{lemma}
\begin{proof}
Knowing that $X$ does not change its fifth coordinate (its boundary conditions as a clock ring would close the window) within $W_0$ allows us to couple the first four coordinates of $X$ with $Y$ while considering both cases of $X$'s boundary conditions and aim for the desired domination. It is not hard to see that $X^1$ dominates $Y^1$ under passive boundary conditions: Note that the transition probabilities as well as the states one ends up in after regress are the same except for state $(n,1,1,1)$, which only slows the progress of $Y$. 

Let us justify why the same holds for healthy boundary conditions by showing that for any path leading to progress in $Y$, we can find a union of heavier paths in $X$ (with those unions being disjoint). It is clear that for all $q <1/2$ and any progressive path in $Y$ last visiting $(n,1,1,0)$, we can find a heavier one in $X$ using the edges between states $(n,1,1,0)$ and $(n,1,1,1)$.

Next, observe that for $X$ with healthy boundary conditions, we have additional edges between states $(n,0,1,1)$ and $(n,0,1,0)$. So any path in $Y$ leading to progress and last visiting either one of these two states may make use of these edges and then progress in $X$. Defining $c:= \tfrac{q(1-q)}{(2+q)(3-q)}$ to be the weight of a 2-cycle between these states we have that, starting from state $(n,0,1,1)$,
		\al{\left( \frac{1-q}{2+q} + c \right)\sum_{k\geq 0} c^k = \frac{1-q}{2},}
so these extra paths add up precisely to the weight of the edges from $(n,0,1,1)$ to $\cs_{n+1}$. An analogous computation gives the same result when starting from state $(n,0,1,0)$. 
\end{proof}
\begin{proof}[Proof of Proposition~\ref{thm:smallq}]
Lemma~\ref{lem:coupling} is valid regardless of the boundary condition, so we can glue together stable windows until the event $X_{R'_1} \notin \cs_n$ is satisfied---that is, until a window ends with a level change. In this case, there are two possibilities, namely $X_{R'_1} \in \cs_{n + 1}$ (progress) or $X_{R'_1} \in \cs_{n-1}$ (regress).

In the canonical coupling of the first four coordinates of $X$ and $Y$ within $W_0$, we obtain that after regress, $X$ and $Y$ end up in the same state (w.r.t.~to the first four coordinates of $X$), whereas after progress, $X$ is in one of the states $(n+1,1,0,1,z), (n+1,1,0,0,z), (n+1,1,1,0,z)$ or $(n+1,1,1,1,z)$ (with $z \in \{0,1\}$ determining the new boundary condition), while $Y$ will find itself in $(n+1,1,0,1)$. So in any case, $Y$ is in a state from which progress in less likely.

In summary, Lemma~\ref{lem:smallqdrift} shows that $Y^0$ dominates a random walk on $\mathbb Z$ with positive drift and so $Y^0$ has a positive drift. Due to the coupling obtained from Lemma~\ref{lem:coupling}, this drift carries over to $X$. Hence, the law or large number for a random walk with drift yields the claimed statement. Finally, Theorem~\ref{mainthm}~(i) follows from Proposition~\ref{thm:smallq} via Claim~\ref{claim:monoinf}.
\end{proof}

\section{Extinction for large $q$} \label{sec:largeq}

We import some notation from Section~\ref{sec:smallq}. Namely, let $(X_i)_{i \in \mathbb N} \subset \mathbb Z \times \{0,1\}^4$ be the discrete time process describing the rightmost infected site and its neighbors and let $\cs_n$ denote all $n$-levels of its state space. Similar to how $\tau$ and $\tau_2$ were defined for the Markov chain $Y$ in that section, we define $\tau_i$ for $i \geq 1$ as
		\al{\tau_{i+1} = \inf\{j\geq \tau_i: \exists n\in \mathbb Z: X_{\tau_i} \in \cs_n, X_j \in \cs_{n \pm 1}  \},}
where we set $\tau_0=0$, to be the sequence of level changes of $X$. We abbreviate $\tau=\tau_1$ when it is convenient. We next define two stopping times describing the length of consecutive progressive and regressive steps, respectively. That is, we set
		\al{\taur &:= \sup \{i \in \mathbb N_0: \exists n \in \mathbb Z: X_0 \in \cs_n, X_{\tau_i} \in \cs_{n+i} \}, \\
				\taul &:= \sup \{i \in \mathbb N_0: \exists n \in \mathbb Z: X_0 \in \cs_n, X_{\tau_i} \in \cs_{n-i}\}}
and call $\taur$ a progressive and $\taul$ a regressive interval, respectively. It is clear that we can partition $X$ into alternating progressive and regressive intervals. Our aim is to prove that the length of a progressive interval is, in expectation, less than the length of a regressive one. Note that if $\tau$ is a regressive step, then $X_\tau \in \mathcal G_n$ for some integer $n$, where
		\al{ \mathcal G_n = \{(n,z_2, 0, z_4, z_5): z_i \in \{0,1\} \textrm{ for } i =2,4,5\} .} 
Similarly, if $\tau$ is a progressive step, then $X_\tau \in \mathcal B_n$ for some $n$, with
		\al{ \mathcal B_n = \{(n,1, z_3, z_4, z_5): z_i \in \{0,1\} \textrm{ for } i =3,4,5\} .} 
The following lemma is the main step in the proof of Theorem~\ref{mainthm} (ii).

\begin{lemma} \label{lem:largeqdrift}
In the above notation, we have that
		\al{\E[\taur \mid X_0 \in \mathcal G_n] < \E[\taul \mid X_0 \in \mathcal B_n]}
for $n \in \mathbb Z$ und $q$ sufficiently large.
\end{lemma}

\begin{proof}[Proof of Theorem~\ref{mainthm} (ii)] As observed above, starting at $\tau$, any progressive interval must start from a $\mathcal G$ state, whereas any regressive interval must start from a $\mathcal B$ state. Hence, the conditioning in Lemma~\ref{lem:largeqdrift} is not a restriction and the rightmost infected site is dominated by a $\mathbb Z$-valued random walk with negative drift, which yields the claimed result.

\end{proof}

\begin{comment}
For the proof of this lemma we shall make use the following corollary of Proposition 4.1 given in \cite{fa1f} by Blondel et al.
\begin{corollary}[Blondel et al.~\cite{fa1f}, Proposition 4.1.]\label{distance}
Given the dynamics of 3SCP1, consider the process on the finite set $\sigma_0=(z_1,\ldots,z_{k}) \in \{0,1\}^{k}$, where we define $z_0=z_{k+1}=0$. Then if $q>2/3$ for any $i \in [k]$ we have that
\begin{align*}
	\mathbb E_{\sigma_0}\left[\min_{j\in [k+1]\cup \{0\}:~ \sigma_t(z_j)=0} \{d(z_j,z_i)\}\right]\leq 2^{(k+2)/2}e^{-(3q/2 -1) t}+\frac{q}{3q-2} \quad \forall t\geq 0.
\end{align*}
\end{corollary}
\begin{proof}
	Observe that we do not have any infections, i.e.~sites in state $2$ and hence the dynamics of our 3SCP1 correspond to those of the FA1f model as introduced in \cite{fa1f}. We set $\theta=2$ and use that $2^x> x$ for all $x$. Also observe that as we have healthy boundary sites, for every $z_i$ the initial minimal distance to an healthy site is bounded by $(k+2)/2$.
\end{proof}
\end{comment}

Turning towards the proof of Lemma~\ref{lem:largeqdrift}, a key observation is the fact that, when $q$ is sufficiently large, healthy sites drift towards each other. More precisely, given a connected set of passive sites with healthy boundary conditions, we expect the size of this set to decrease with time. With this in mind we define
		\al{\xi^x(\eta) = \inf \{|y-x|: y \in \mathbb Z, \eta(y) = 0\},}
for some $\eta \in \Omega$ and $x \in \mathbb Z$, i.e.~the distance of $x$ to the next healthy site in $\eta$.

\begin{lemma} \label{lem:constpstrips}
	Let $q > 1/2$. Given any initial distribution $\nu$ taking values in $\{0,1\}^\mathbb Z$. Assume that $\kappa := \E[\xi^x(\nu)] < \infty$ for some site $x \in \mathbb Z$. Then for the process $(\eta_t)$ with $\eta_0 \sim \nu$, we have
	\al{\mathbb E^{\nu}\left[\xi^x(\eta_t)\right] \leq \max\{1, \kappa + t(1-2q)\} \quad \forall t\geq 0.}
\end{lemma}
\begin{proof}
Let $\eta = \eta_0$ be state of the process at time $0$. Due to translation invariance and symmetry, we shall consider site $x=0$ and assume the closest healthy site is located at $\xi_t = \xi^0(\eta_t) \gg 1$ for all times $t$. Since we are only interested in an upper bound, we always assume that $\eta(\xi_t+1) = 0$. In doing this, we obtain a process whose $\xi_t$-value dominates the original one. We thus end up with the following simplification:
\begin{itemize}
	\item If site $\xi_t$ updates, then with probability $1-q$, it becomes passive and $\xi_{t^+} = \xi_t + 1$,
	\item if site $\xi_t-1$ updates, then with probability $q$, it becomes healthy and $\xi_{t^+} = \xi_t - 1$,
\end{itemize}
and those are the only updates changing the position of $\xi_t$. Hence, the expected change of $\eta_t$ after an update is $1-2q<0$. The number of updates in $[0,t]$ of these two sites is $2\text{ Poisson}(t)$-distributed, and with probability $1/2$, an update yields a change of position, so $N_t$, the number of position changes in $[0,t]$, is $\text{Poisson}(t)$-distributed. Hence, as all of this remains true for $\xi_t \geq 1$, the statement follows by Wald's lemma.
\end{proof}
Note that Lemma~\ref{lem:constpstrips} is very much in the spirit of Proposition~4.1 in \cite{fa1f}, even though we need a much weaker statement to prove Lemma~\ref{lem:largeqdrift}, namely that $E^{\nu}\left[\xi^x(\eta_t)\right]$ is not increasing.
\begin{proof}[Proof of Lemma \ref{lem:largeqdrift}]
We begin by considering $\taul$ and noting that, no matter the boundary conditions,
		\al{\p[X_{\tau_2} \in \cs_{n-1} \mid X_{\tau_1} \in \cs_n, X_0 \in \cs_{n+1}] &= \p[X_\tau \in \cs_{n-1} \mid X_0 \in \mathcal G_n] \geq \alpha (q) \\
				 & = \min\left\{ \frac{q}{2-q} \cdot \frac{1}{4-3q}, \frac{1}{3-q} (1+\frac{q}{4-3q})  \right\} \xrightarrow{q \to 1} 1.}
In words, following a regressive step, we witness another regressive step with probability at least $\alpha \to 1$. That is because from $\mathcal G_n$, $X$ ends up in another regressive step within three steps or less, regardless of a change of boundary conditions during that time. As a direct consequence, $\E[\taul \mid X_0 \in \mathcal G_n] \geq (1-\alpha(q))^{-1}$ gets arbitrarily large for $q \to 1$. On the other hand, we have that there exists $\beta<1$ such that
		\al{\p[X_\tau \in \cs_{n+1} \mid X_0 \in \mathcal B_n \backslash\{(n,1,1,1,1) \}] \leq \beta}
for all $q$ not too small ($q > 1/2$ say), and thus
		\al{\E\left[ \taur \mid X_0 \in \mathcal B_n \backslash\{(n,1,1,1,1)\} \right] \leq \beta \left(1+\E\left[ \taur \mid X_0 \in \mathcal B_n \right]\right) \leq \beta \left(1 + \E\left[ \taur \mid X_0 = (n,1,1,1,1) \right]\right).}
So if we can bound the last quantity by some constant, we are done. This is where Lemma~\ref{lem:constpstrips} comes in. We bound this expectation by ``jumping'' to the closest healthy site, infecting all passive sites on the way. More precisely, we progress the infection by force until reaching a state in $\mathcal G_n$.
		\al{\E\left[ \taur \mid X_0 = (n,1,1,1,1) \right] & \leq \sum_{i=0}^\infty \left( \E\left[ \taur \mid X_0 \in \mathcal G_n, \xi^{\I(0)}(\eta_0) = i+2 \right] + i \right) \p\left[\xi^{\I(0)}(\eta_0) = i+2\right] \\
				& \leq \sum_{i=0}^\infty i \p\left[\xi^{\I(0)}(\eta_0) = i\right] +  \E\left[ \taur \mid X_0 \in \mathcal G_n\right] \sum_{i=0}^\infty \p\left[\xi^{\I(0)}(\eta_0) = i+2\right]  \\
				& \leq \E\left[\xi^{\I(0)}(\eta_0) \right] + \E\left[ \taur \mid X_0 \in \mathcal G_n\right], }
which is bounded by a constant, combining Lemma~\ref{lem:constpstrips} with the fact that the second term goes to $0$ as $q\to 1$.
\end{proof}

\bibliography{bibliography}{}
\bibliographystyle{amsplain}

\end{document}